\documentclass[a4paper,11pt]{amsart}
\usepackage{graphicx,amssymb}

\newtheorem{thm}{Theorem}[section]

\newtheorem{lem}[thm]{Lemma}
\newtheorem{prop}[thm]{Proposition}
\newtheorem*{thm A}{Theorem A}
\newtheorem*{thm B}{Theorem B}
\theoremstyle{definition}

\newtheorem{question}[thm]{Question}

\theoremstyle{remark}
\newtheorem{rem}[thm]{Remark}
\numberwithin{equation}{section}

\newcounter{numl}

\newcommand{\labelnuml}{\textup{(\roman{numl})}}
\newenvironment{numlist}{\begin{list}{\labelnuml}%
{\usecounter{numl}\setlength{\leftmargin}{0pt}%
\setlength{\itemindent}{2\parindent}%
\setlength{\itemsep}{\smallskipamount}\def
\makelabel ##1{\hss \llap {\upshape ##1}}}}{\end{list}}

\newcommand{\del}{\partial}
\def\cal{\mathcal}
\def\bb{\mathbb} 
\def\frak{\mathfrak }
\def\a{\alpha } 

\def\g{\gamma }

\def\g{\gamma }

\def\o{\omega }

\def\s{\sigma }

\def\t{\theta }

\def\ot{\otimes }
\def\part{\partial }
\def\bpart{\bar\partial }

\def\w{\wedge }

\begin{document}

\title[From lcK to bi-Hermitian] {From locally conformally K\"ahler to bi-Hermitian structures on non-K\"ahler complex surfaces}

\author[V. Apostolov]{Vestislav Apostolov} \address{Vestislav Apostolov \\
D{\'e}partement de Math{\'e}matiques\\ UQAM\\ C.P. 8888 \\ Succ. Centre-ville
\\ Montr{\'e}al (Qu{\'e}bec) \\ H3C 3P8 \\ Canada}
\email{apostolov.vestislav@uqam.ca}

\author[M. Bailey] {Michael Bailey}\address{Michael Bailey \\
D{\'e}partement de Math{\'e}matiques\\ UQAM\\ C.P. 8888 \\ Succ. Centre-ville
\\ Montr{\'e}al (Qu{\'e}bec) \\ H3C 3P8 \\ Canada}
\email{bailey@cirget.ca}

\author[G. Dloussky]{Georges Dloussky}\address{Georges Dloussky \\ Centre de Math\'ematiques et Informatique\\ 
LATP, UMR 6632 \\
Universit\'e d'Aix-Marseille \\
Technop\^ole Ch\^ateau-Gombert \\
39, rue F. Joliot Curie \\
13453 Marseille Cedex 13 \\ France}
\email{georges.dloussky@univ-amu.fr}

\thanks{V.A. was supported in part by an NSERC discovery grant and is grateful to
the Simons Center for Geometry and Physics  and the Institute of Mathematics and Informatics of the Bulgarian Academy of
Sciences where a part of this project was realized. V.A. and G. D. were supported in part by the ANR grant  MNGNK-ANR-10-BLAN-0118. The authors are grateful to M. Gualtieri and R. Goto for their interest and valuable suggestions as well as the anonymous referees for the  careful reading of the initial manuscript and  constructive suggestions which greatly improved the presentation.}

\date{\today}

\begin{abstract} We prove that locally conformally K\"ahler metrics on certain compact complex surfaces with odd first Betti number can be deformed to new examples of bi-Hermitian metrics.
\end{abstract}

\maketitle


\section{Introduction}

A \emph{bi-Hermitian} structure on a complex manifold $S=(M,J)$ consists of a pair $(J_+=J,J_-)$ of integrable complex structures, inducing the same orientation, each of which is orthogonal with respect to a common Riemannian metric $g$.  We are generally only interested in the conformal class $c=[g]$.  Furthermore, the case when $J_+ \equiv J_-$ or $J_+ \equiv - J_-$ is considered trivial, so we shall also assume that $J_+(x) \neq \pm J_-(x)$ for at least one point $x\in M$.

Bi-Hermitian geometry has attracted a great deal of interest recently through its link with \emph{generalized K\"ahler geometry},  a natural  extension of K\"ahler geometry  first  studied by Gualtieri~\cite{gualtieri-GK} in the context of generalized complex structures introduced by N. J. Hitchin~\cite{hitchin0}. It is shown in \cite{gualtieri-GK} that a generalized K\"ahler structure is equivalent to the data of a bi-Hermitian structure $(g,J_+, J_-)$, satisfying the relations 
\begin{equation}\label{relations}
d^c_+ F_+ =- d^c_-F_- = dB, 
\end{equation}
for some 2-form $B$, where $F_{\pm}(\cdot, \cdot)=g(J_{\pm}\cdot, \cdot)$ are the corresponding fundamental $2$-forms of the Hermitian structures $(g, J_{\pm})$, and $d^c_{\pm}=i(\bar\partial_{\pm}-\partial_{\pm})$ are the associated complex operators. We may (trivially) represent a K\"ahler structure $(J, \omega)$ by taking $J_{\pm}= \pm J,  F_{\pm}= \pm \omega$ while recent work of Goto~\cite{Goto-JDG} provides a way to deform K\"ahler metrics to non-trivial generalized K\"ahler structures, i.e. bi-Hermitian structures satisfying \eqref{relations},  and for which $J_+ \neq - J_-$ at at least one point of $M$.

\vspace{0.2cm}
This work is a part of the larger problem of the existence of (conformal classes of) bi-Hermitian structures on compact complex surfaces. In this case, to each bi-Hermitian structure $(c, J_+ , J_-)$ on $S=(M,J=J_+)$, one can associate (using the commutator $[J_+, J_-]= J_+J_- - J_+ J_-$ and a reference metric $g \in c$) a non-trivial holomorphic section $\sigma= [J_+, J_-]^{\sharp} \in H^0(S, {\cal K}^*_S\otimes {\mathcal L})$ of the anti-canonical bundle ${\cal K}^*_S$ of $S$,  twisted with a topologically trivial flat holomorphic line bundle ${\mathcal L}$, see~\cite[Lemma~3]{AGG}. Furthermore, bi-Hermitian structures  $(c, J_+, J_-)$ on compact real $4$-dimensional (connected) manifolds $M$ can be divided into three different classes as follows:
\begin{enumerate}
\item[(i)] Everywhere on $M$, $J_+ \neq J_-$ and $J_+ \neq -J_-$;
\item[(ii)] Everywhere on $M$, $J_+ \neq J_-$ (resp. $J_+ \neq -J_-$), but for at least one $x \in M$, $J_+(x) = -J_-(x)$ (resp. $J_+(x) = J_-(x)$, though---by replacing $J_-$ with $-J_-$ if necessary---we can assume without loss of generality that in this class $J_+$ and $J_-$ never agree but $J_+$ and $-J_-$ sometimes do);
\item[(iii)] There are points on $M$ where $J_+ = J_-$ and also points where $J_+ = - J_-$.
\end{enumerate}

Recall~\cite{siu, todorov, buchdahl, lamari} that on a compact complex surface $S=(M,J)$ a K\"ahler metric exists if and only if the first Betti number is even. Similarly, by  \cite[Cor.~1 and Prop.~4]{AG}, a bi-Hermitian conformal structure $(c,J_+, J_-)$ corresponds to a generalized K\"ahler structure for some $g\in c$ if and only if $b_1(M)$ is even. Furthermore, in this case the flat holomorphic line bundle ${\mathcal L}$ mentioned above is trivial (\cite[Lemma 4]{AGG}) and the bi-Hermitian structures are either of type (i) or (ii) (\cite[Prop.~4]{AGG}).  The first case corresponds to K\"ahler surfaces with trivial canonical bundle (see \cite{AGG}), i.e. tori and $K3$ surfaces.  The classification in the second case follows by  \cite{AGG,BM} and a recent result in \cite{Goto-BH}: $S$ must be then a K\"ahler surface of negative Kodaira dimension whose anti-canonical bundle ${\cal K}^*_S$ has a non-trivial section and any K\"ahler metric on  $S=(M,J_+)$ can be deformed to a non-trivial bi-Hermitian structure $(c,J_+, J_-)$ of the class (ii).

In the case when $S$ doesn't admit K\"ahler metrics (i.e. the first Betti number of $S$ is odd),  the complex surfaces supporting bi-Hermitian structures in the class (i)  are classified  in \cite{AD}.

Finally, another case for which the existence theory is fairly complete by~\cite{FP,CG} consists of the bi-Hermitian complex surfaces arising from {\it twisted generalized K\"ahler} structures, i.e., those for which relation \eqref{relations} is weakened to $d^c_+ F_+ =- d^c_-F_-= H$ for some closed $3$-form $H$: when the de Rham class $[H] \in H^3_{dR}(M)$ is trivial,  we recover the generalized K\"ahler case discussed above, while when $[H]\neq 0$ one gets bi-Hermitian structures with ${\mathcal L} \cong \mathcal O$  of the class (iii)  (\cite[Prop.~4]{AGG})  on compact complex surfaces in the Kodaira class VII (\cite[Thm.~1]{A}).

\vspace{0.2cm}

Thus motivated, in this note we narrow our focus to the existence of compatible bi-Hermitian structures of the class (ii) on compact complex surfaces $S=(M,J)$ with odd first Betti number. It is shown in \cite{A} that $S$ then must be a complex surface in the Kodaira class VII (i.e. $S$ has Kodaira dimension $-\infty$ and $b_1(S)=1$) while \cite{Dl} provides a complete list of possibilities for the minimal model of $S$.  A more exhaustive taxonomy of bi-Hermitian complex surfaces with odd first Betti number is provided in the appendix A.

\vspace{0.2cm}  One may regard a general bi-Hermitian structure $(c, J_+, J_-)$ on a compact $4$-manifold $M$ as relaxing the generalized K\"ahler compatibility relation \eqref{relations}.  Even when $b_1(M)$ is odd, a choice of metric in $c$ satisfying \eqref{relations} exists \emph{locally} (see~ \cite[Lemma~1]{AGG} and \cite[Prop.~6]{CG}); thus, compatible bi-Hermitian conformal classes on $S$ are always locally conformal to generalized K\"ahler structures.  It turns out that under the assumption (ii), one can further relate the bi-Hermitian structures to locally conformally K\"ahler metrics, in a similar way that non-trivial generalized K\"ahler structures arise as deformations of genuine K\"ahler ones \cite{Goto-BH, Goto-JDG}.  This is the context for our main result.

\vspace{0.2cm} Recall that a {\it locally conformally K\"ahler} (or \emph{lcK}) metric on a complex manifold $S=(M,J)$ may be defined by a positive-definite $(1,1)$-form $F$  satisfying 
$dF= \theta \wedge F$ for a closed $1$-form $\theta$. The $1$-form $\theta$ is uniquely determined and is referred to as {\it the Lee form} of $F$. The corresponding Hermitian metric $g(\cdot, \cdot)= F(\cdot, J \cdot)$ defines a conformal class  $c$ on $M$. Changing the Hermitian metric $\tilde g = e^f g$ within $c$ amounts to transforming the Lee form by $\tilde \theta = \theta + df$, showing that the de Rham class $[\theta]$ is an invariant of the conformal class $c$.   The study of lcK metrics, which goes back to foundational works by F. Tricerri and I. Vaisman, is a natural extension of the theory of K\"ahler metrics to certain classes of non-K\"ahlerian complex manifolds, see e.g.~\cite{DO, OV} for an overview of the theory.  Of particular interest is the case of compact complex surfaces, where recent works~\cite{B,B1,B2,GO} showed that lcK metric exists for all known (and conjecturally for all) compact complex surfaces with odd first Betti number,  with the one exception of certain Inoue surfaces with zero second Betti number described in \cite{B}.  

Let  $S=(M,J)$ be a compact complex surfaces in the class VII.   By the well-known isomorphism (see e.g. \cite{bpv})
\begin{equation}\label{H1}
H^1_{dR}(S,\bb C)\simeq H^1(S,\cal O)\stackrel{\exp }{\longrightarrow}Pic^0(S)\simeq H^1(S,\bb C^\star)\simeq \bb C^\star
\end{equation}
for any de Rham class $a\in H^1_{dR}(S,\bb C)$ there exists a unique flat holomorphic line bundle  ${\mathcal L}_{a}$ over $S$. In the case where $a$ is real, i.e., where it belongs to $H^1_{dR}(S,\bb R)$, ${\mathcal L}_{a}$ is the complexification of a real  flat bundle $L_{a}$ over $S$, and in the sequel we will tacitly identify ${\mathcal L}_{a}$ with $L_{a}$, referring to such flat holomorphic bundles as being of real type.  Then we can make the following conjecture:

\vspace{0.2cm}
\noindent
{\bf Conjecture.} \ Let $S=(M,J)$ be a compact complex surface in the class VII such that $H^0(S,{\cal K}_S^{*}\ot {\mathcal L})\neq 0$ for a  flat holomorphic bundle of real type ${\mathcal L}$ with $H^0(S, {\mathcal L^{\ell}})=0$ for all $\ell \ge 1$. Then the following two conditions are equivalent:
\begin{enumerate}
\item[$\bullet$] There exists a bi-Hermitian structure $(g,J_+,J_-)$ of the class (ii) on $(M,J)$,  such that $J=J_+$ and  $\sigma= [J_+, J_-]^{\sharp} \in H^0(S, {\cal K}^*_S\otimes \mathcal L)$.
\item[$\bullet$]  There exists a lcK metric  with Lee form $-\t$ whose de Rham class in $H^1_{dR}(S,{\bb C})$ corresponds to the flat bundle ${\mathcal L}^{*}$.
\end{enumerate}
The assumptions are justified by the fact that, by \cite[Proposition~4]{AGG}, \cite[Theorem 1]{A}, and the degree computation of \cite[p.~561]{A},  the two cohomological conditions in the above Conjecture are  necessary for the existence of a bi-Hermitian metric satisfying (ii), while $H^0(S, {\mathcal L^{\ell}})=0$ is necessary for the existence of a lcK metric with Lee form corresponding to $L^*$.

\vspace{0.2cm}
We will establish one direction of the conjectured correspondence by extending, from the K\"ahler case to the strictly lcK case, certain deformation arguments due to R.~Goto~\cite{Goto-BH}, N.~J.~Hitchin~\cite{Hitchin-deform}  and M.~Gualtieri~\cite{Gual}.

\begin{thm}\label{main}  Let  $S=(M,J)$ be a compact complex surface in the class VII such  that $H^0(S,{\cal K}_S^{*}\ot {\mathcal L})\neq 0$ for a flat holomorphic line bundle of real type ${\mathcal L}$. Let $a\in H^1_{dR}(S, \bb R)$ be the real de Rham class corresponding to ${\mathcal L}$ and suppose that $S$ admits a lcK metric $g$ with Lee form in $a$. Then $S$ also admits a bi-Hermitian conformal structure $(c, J_+, J_-)$ with $J_+=J$ and $\sigma = [J_+, J_-]^{\sharp} \in H^0(S,{\cal K}_S^{*}\ot {\mathcal L})$.
\end{thm}

We use Theorem~\ref{main} to give new examples of bi-Hermitian metrics in the class (ii) on certain Hopf surfaces.

\section{Preliminaries}
For a closed $1$-form $\theta$ on $M$, we denote by $L=L_{\theta}$ the flat real line bundle determined by the class $[\theta] \in H^1_{dR}(M)$ and by $L^*$ its dual. The differential operator  $d_\t=d- \t\wedge$ then defines the Novikov complex
$$\cdots \stackrel{d_\t}{\to}\Omega^{k-1}(M) \stackrel{d_\t}{\to} \Omega^{k}(M)  \stackrel{d_\t}{\to}\cdots$$
and the corresponding cohomology groups $H^{k}_{\theta}(M)$. Let $\frak U=(U_i)$ be an open covering such that $\t_{\mid U_i}=df_i$.  Then, $\frak U$ defines a trivialization for $L$ with (constant) transition functions $e^{f_i-f_j}$ on $U_{ij}=U_i\cap U_j$; furthermore,  $(U_{i}, e^{-f_i})$ defines an isomorphism, denoted by  $e^{-f}$,  between the Novikov complex and $L^*$-valued de Rham complex
 $$\cdots \stackrel{d_{L^*}}{\to}\Omega^{k-1}(M,L^*) \stackrel{d_{L^*}}{\to} \Omega^{k}(M, L^*)  \stackrel{d_{L^*}}{\to}\cdots$$
 which acts at degree $k$ by $e^{-f}(\a)=(e^{-f_i}\a_{\mid U_i})$ for any $\a \in \Omega^k(M)$, and  thus
$$d_\t=e^{f_i} {d_{L^*}}_{\mid U_i} e^{-f_i}.$$
In particular,  we have an isomorphism between the cohomology groups
$$H^{k}_\t(M)\simeq H^k(M,L^*).$$

Considering the Dolbeault cohomology groups of $S$ with values in the flat holomorphic line bundle ${\mathcal L}^*= {L^*}\otimes {\bb C}$, we have 
$$d_{L^*}=\part _{\mathcal {L^*}}+\bpart_{\mathcal {L}^*}, \quad {\rm and}\quad d_\t=\part_\t+\bpart_\t$$
 with 
 $$\part_\t=\part-\t^{1,0}\w\quad {\rm and}\quad \bpart_\t=\bpart-\t^{0,1}\w,$$
giving rise to the isomorphisms 
$$H^{p,q}_{\bar \partial_\t}(S)\simeq H^{p,q}(S,\cal L^*).$$

Similarly, the space of holomorphic  sections $H^0(S, {\cal K}_S^{*}\otimes {\cal L})$ can be naturally identified with the space of smooth sections of $\bigwedge^2 (T^{1,0}M)$ in the kernel of the twisted Cauchy--Riemann operator 
\begin{equation*}
\bar \partial_{\theta}  \sigma = \bar \partial \sigma + \theta^{0,1}\otimes \sigma.
\end{equation*}

\vspace{0.2cm}
We shall use the following vanishing result.
\begin{prop}\label{vanish} Let $S$ be a compact complex surface in the class VII and ${\cal L}$ a flat holomorphic line bundle over $S$, such that $H^0(S, {\cal K}^*_S\otimes \cal L)\neq 0$ and $H^0(S, \cal L^{\otimes 2})=0$. Then  $H^{0,2}(S, \cal L^*)=0$. In particular,  for any $(0,2)$-form with values in ${\cal L^*}$, $\alpha$, there exists a $(0,1)$-form with values in ${\mathcal L^*}$, $\beta$,  such that $\alpha = \bar \partial_{\cal L^*} \beta$.
\end{prop}
\begin{proof} As $H^0(S, {\cal L}^{\otimes 2})=0$ and $H^0(S, {\cal K}^*_S\otimes {\cal L})\neq 0$, it follows that $H^0(S, {\cal K}_S\otimes \cal L)=0.$ By Serre duality, $H^2(S, {\cal L}^*)\cong H^{0,2}(S, {\cal L}^*)=0$. As $\Omega^{0,3}(M, {\cal L}^*)=0$, $\alpha= \bar \partial_{{\cal L}^*} \beta.$ \end{proof}

\begin{rem}\label{r:2} {\rm If $S$ is a minimal complex surface in the class VII with $b_2(S)>0$ and $H^0(S, {\cal K}^*_S \otimes \cal L)\neq 0$ for some flat line bundle $\mathcal{L}$, then $H^{0,2}(S, \cal F)=0$ for any non-trivial flat line bundle $\cal F$ by \cite{Dl}, \cite[Lemma 2.1]{DOT3} and Serre duality.}
\end{rem}

\vspace{0.2cm}
Our proof  of Theorem~\ref{main} will rely on the following proposition, which should be regarded as a straightforward generalization of \cite[Theorem~6.2]{Gual} to the case of \emph{locally conformal generalized K\"ahler structures}~\cite{Vaisman}.
\begin{prop}\label{marco}{\rm \cite{Gual}}\label{twisted lemma of Gualtieri} Suppose $S=(M,J)$ is a compact complex surface as in Theorem~\ref{main}, endowed with a holomorphic section  $\s\in H^0(S,{\cal K}_S^*\ot {\mathcal L})$, where ${\mathcal L}$ is a flat holomorphic line bundle of real type corresponding to a de Rham class $[\t]\in H^1_{dR}(M,\bb R)$. Let $Q={\rm Re} (\s )$, and $\o\in \Omega^2(M,L^{*})$ be a $d_{L^*}$-closed 2-form with values in $L^*$ such that
\begin{enumerate}
\item[$\bullet$] the $J$-invariant part of $\omega$ is positive definite;
\item[$\bullet$] $\o J - J^*\o + \o Q \o=0$, where $\o:TM\to T^*M\ot L^{*}$, $Q:T^*M\to TM\ot L$ and $J^*$ acts on $T^*M$ by $J^*\alpha (\cdot) = -\alpha (J \cdot)$.
\end{enumerate}
Then, 
\begin{enumerate}
\item[\rm (i)]   $J_-:=-J -Q\o$ is an integrable complex structure on $M$;
\item[\rm (ii)] $g=-\frac{1}{2}\o(J - J_-)$ is a symmetric tensor field with values in $L^*$ which defines a conformal class $c=[g]$ of Riemannian metrics on $M$;
\item[\rm (iii)]$J_+:=J$ and $J_-$ are orthogonal with respect to $c$ and $J_+(x) \neq J_-(x)$ on $M$.
\end{enumerate}
\end{prop} 
\begin{proof} Let $\frak U=(U_i)$ be an open covering such that $\t_{\mid U_i}=df_i$. This defines a trivialization for $L$ with positive constant transition  functions  $(U_{ij}, e^{f_i-f_j})$. We can then write $\omega=(\omega_i)$ and $\sigma = (\sigma_i) $ with respect to $\frak U$, with $\omega_i$ (resp $\sigma_i$) being closed 2-forms (resp. holomorphic Poisson structures) on each $U_i$ such that $\omega_j = e^{f_j-f_i}\omega_i$ (resp. $\sigma_j = e^{f_i-f_j} \sigma_i$). Putting $Q_i= {\rm Re} (\sigma_i)$,  by \cite[Theorem~6.2]{Gual}  the $J$-invariant part of $\omega_i$ gives rise to a bi-Hermitian metric  $g_i$ on each $U_i$  with  $g_j = e^{f_j-f_i} g_i$. Thus, the conformal structures $([g_i], U_i)$ extend to a global conformal class of Riemannian metrics on $M$; similarly,  $(U_i, Q_i\omega_i)$ is a well-defined tensor field  on $M$, showing that $J_-$ is  an integrable almost complex structure on $M$. Finally, in order to verify (iii), suppose that $J_+(x)=J_-(x)$ for some $x\in M$.  It follows from (i) that the endomorphism $Q \o$ of $T_xM$ commutes with $J$. As the $J$-invariant part of $\omega$ is positive-definite (and therefore non-degenerate) while $Q$ anti-commutes with $J$, one concludes that $Q$ must vanishes at $x$. But then, according to (i),  $J_- (x) = -J(x) = -J_+(x)$, a contradiction. \end{proof}

\vspace{0.2cm}
Conversely, the general theory of bi-Hermitian complex surfaces~\cite{AGG} implies
\begin{prop}\label{BH} Any bi-Hermitian structure $(c,J_+, J_-)$ on a compact $4$-manifold $M$,  such that $J_+(x) \neq J_-(x)$ for each $x\in M$ and $J_+(x) \neq -J_-(x)$ for at least one point $x$,  arises from Proposition~\ref{marco}.
\end{prop}
\begin{proof} With respect to a reference metric $g \in c$, let $F_{\pm}(\cdot, \cdot)= g(J_{\pm}\cdot, \cdot)$ denote the fundamental $2$-forms of the Hermitian structures $(g, J_+)$ and $(g, J_-)$, respectively, and $\theta_{\pm}$ the corresponding Lee forms defined by $dF_{\pm} = \theta_{\pm}\wedge F_{\pm}$.  By \cite[Lemma~1]{AGG}, $\theta_+ + \theta_-$ is closed (as $M$ is compact). Let $J$ denote one of the complex structures, $J_+$ say, and $S=(M,J)$ the corresponding complex surface. By \cite[Lemma~3]{AGG}, the $(1,1)$ tensor
\begin{equation*}
\Phi:= \frac{1}{2}(J_+  J_- - J_- J_+)
\end{equation*}
can be transformed, via the metric $g$,  to a $(2,0)$ tensor $Q$. The later defines a smooth section $\sigma$ of $\bigwedge^2(T^{1,0}S)$ with ${\rm Re}(\sigma)=Q$, which belongs to the kernel of $\bar \partial_{\theta}= \bar \partial + \theta^{0,1} \otimes$ with $\theta=  \frac{1}{2}(\theta_+ + \theta_-)$. Thus, by the discussion at beginning of this section, $\sigma$ can be equally seen as an element of  $H^0(S, {\cal K}^*_S\otimes \cal L)$ where $\cal L= L\otimes {\mathbb C}$ is the flat holomorphic line bundle of real type, corresponding to  the  de Rham class $[\theta]$.

Letting $p= - \frac{1}{4}{\rm trace} (J_+ J_-)$, one has (see e.g.~\cite[Eq.~(2)]{AGG}) that $p$ is a smooth function on $M$ with values in $[-1,1]$. Furthermore, $p(x)= \pm 1$ if and only if $J_+(x) = \pm J_-(x)$. Thus, our assumption is that $p<1$ on $M$, so that the $2$-form
\begin{equation}
\omega(\cdot, \cdot) := F_+ (\cdot, \cdot) - \frac{1}{1-p} g(\Phi  J \cdot, \cdot),
\end{equation}
is well-defined on $M$ and is manifestly self-dual with respect to $g$. The co-differential  of $\omega$ has been computed in the proof of \cite[Proposition~4]{AGG} to be $\delta^g \omega (\cdot ) = -\frac{1}{2}\omega((\theta_+ + \theta_-)^{\sharp}, \cdot), $ where $\sharp$ stands for  the vector field corresponding to a $1$-form via the metric $g$. As $\omega$ is self-dual, the last equality equivalently reads as
\begin{equation*}
d\omega = \frac{1}{2}(\theta_+ + \theta_-) \wedge \omega,
\end{equation*}
showing that $\omega$ is $d_{\theta}$-closed and thus  can be identified with a $d_{L^*}$-closed $2$-form with values in $L^*$  (see the discussion at the beginning of the section). Note that the $J$-invariant part of $\omega$ is, by construction, the positive-definite fundamental $2$-form $F_+$.  The relations (i) and (ii) between $J_+, J_-, g$ and $\omega$ are checked easily. \end{proof}

\section{Proof of Theorem~\ref{main}}

We start with a compact complex surface $S=(M,J)$ in the class VII, endowed with a section $\s \in H^0(S, {\cal K}^*_S\otimes \cal L)$, for a flat holomorphic line bundle $\cal L= L \otimes {\mathbb C}$, where $L$ is a flat real line bundle corresponding to a class $[\theta]\in H^1_{dR}(S)$. Let $F$ be the fundamental form of a lcK metric on $S$ with Lee form $\theta$.  Thus,  $F$ is $d_{\theta}$ closed,  positive definite $J$-invariant $2$-form on $S$, which we will also identify with a $d_{L^*}$-closed positive definite $J$-invariant $2$-form with values in $L^*$ (still denoted by $F$).  Note that the degree of $\cal L$ with respect to a Gauduchon metric $g$ in the conformal class of the lcK structure $F$ is ${\rm deg}_g (\cal L)=- \frac{1}{2\pi}\int_M |\theta|^2 dv_g <0$ (see e.g. \cite[Eq.~(5)]{AD}), so  that $H^0(S, \cal L^{\otimes \ell})=0, \  \forall \ell \ge 1$. It follows that for any $(0,2)$-form with values in $\cal L^*$,  $\alpha = \bar \partial_{\cal L^*} \beta$,  see Proposition~\ref{vanish}. 

We wish to find a family $\omega(t)$ of $d_{L^*}$-closed 2-forms with values in $L^*$, such that $\omega(0) = 0$, $\dot{\omega}(0)= F$ (where the dot stands for the derivative with respect to $t$) and, for sufficiently small $t>0$, $(\omega(t), Q= {\rm Re}(\sigma))$ satisfy the two conditions of Proposition~\ref{marco}.  Note that the boundary condition at $t=0$ for $\omega(t)$ implies that the $J$-invariant part of $\omega(t)$ will be positive definite for $t>0$ sufficiently small, so we have to deal with the second condition relating $\omega(t)$ and $Q$. To this end,  we suppose $\omega(t)$ is expressed as a power series in $t$,
$$\omega(t) = t\omega_1 + t^2\omega_2 + \ldots,$$
where each $\omega_n$ is a $d_{L^*}$-closed real 2-form with values in $L^*$ and $\omega_1 = F$.

The equation
\begin{equation}\label{B condition}
\o(t) J - J^*\o(t)+ \o(t) Q \o(t)=0
\end{equation}
relates to the $(2,0)+(0,2)$ part of $\omega(t)$.  In other words, it may be expressed as
\begin{equation}\label{e1}
2(J^*   \omega(t)^{2,0 + 0,2} ) - \omega(t)  Q   \omega(t) = 0.
\end{equation}
If we decompose this term-by-term, we have (factoring out $t^n$)
\begin{equation}\label{term condition}
J^*   \omega_n^{2,0+0,2} = \frac{1}{2} \sum_{i+j = n} \omega_i  Q  \omega_j.
\end{equation}
Since $\omega_1 = F$ is $(1,1)$, this is satisfied for $n=1$.  Given $\omega_i$ for all $i < n$, \eqref{term condition} fixes $\omega_n^{2,0+0,2}$.

Since we need that $d_{L^*} \big(\omega(t)\big)=0$, in particular  we must have $\bar \del_{\cal L^*} \omega_n^{0,2} = 0$ for the $\omega_n^{0,2}$ thus determined: in  complex dimension $2$  this is automatic.  By Proposition~\ref{vanish}, there exists a $(0,1)$-form with values in ${\cal L^*}$, $\beta_n$,  such that
\begin{equation}\label{beta construction}
\del_{\cal L^*} \omega_n^{0,2} = \del_{\cal L^*}\bar\del_{\cal L^*} \beta_n.
\end{equation}
Letting
\begin{equation}\label{1,1}
\omega_n^{1,1} = \del_{\cal L^*} \beta_n + \bar\del_{\cal L^*}\bar\beta_n, 
\end{equation}
the $L^*$-valued real $2$-form  $\omega_n$ defined by \eqref{term condition} and \eqref{1,1} satisfies (by using \eqref{beta construction})
\begin{equation*}
\begin{split}
d_{L^*} \omega_n &= (\partial_{\mathcal{L}^*} + \bar \partial_{\mathcal{L}^*}) (\omega_n^{2,0} + \omega_n^{0,2}  + \omega_n^{1,1}) \\
                              &=  \partial_{\mathcal{L}^*} \omega_n^{0,2} + \bar \partial_{{\mathcal L}^*}\omega_n^{2,0} + \partial_{\mathcal{L}^*} \omega_n^{1,1} +  \bar \partial_{{\mathcal L}^*} \omega_n^{1,1} \\
                             &= \partial_{\mathcal{L}^*} \bar \partial_{\mathcal{L}^*} \beta_n  +  \bar \partial_{\mathcal{L}^*}  \partial_{\mathcal{L}^*} \bar \beta_n   + \bar \partial_{\mathcal{L}^*}  \partial_{\mathcal{L}^*} \beta_n + \partial_{\mathcal{L}^*} \bar \partial_{\mathcal{L}^*}  \bar \beta_n \\
                              &=0,
\end{split}
\end{equation*}                              
as required.

\bigskip
Thus, in order to show that our choices of $\omega_i$ satisfy equation \eqref{term condition}, it remains to be shown that the $(1,1)$-part of \eqref{term condition} vanishes, i.e.,
\begin{lem}\label{1,1 vanish lemma}
Given the $\omega_n$ as defined above,
\begin{align}
\sum_{i+j = n} \left(\omega_i  Q  \omega_j\right)^{1,1} \label{1,1 vanish} = 0.
\end{align}
\end{lem}

\begin{proof}
In other words, we want to show that
$$\sum_{i+j = n} \left(J^* \omega_i  Q  \omega_j + \omega_i  Q  \omega_j  J\right)=0.$$
Since $Q$ anti-commutes with $J$, i.e. $J  Q = -Q  J^*$, this is
\begin{align*}
 &\sum_{i+j = n} \left(J^*  \omega_i  Q   \omega_j -  \omega_i  J  Q   \omega_j - \omega_i  Q   J^*  \omega_j + \omega_i  Q   \omega_j  J \right) \\
=& \sum_{i+j = n} (J^*  \omega_i - \omega_i  J)  Q   \omega_j - \sum_{i+j = n} \omega_i  Q   (J^*  \omega_j - \omega_j  J).
\end{align*}
Now let us assume, inductively, that \eqref{1,1 vanish}---or, equivalently, \eqref{term condition}---holds for all $n' < n$.  Then we make the substitution
$$J^*  \omega_i - \omega_i  J = \sum_{k+l = i} \omega_k  Q  \omega_l$$
(and likewise for $\omega_j$), so that we finally  get
\begin{align*}
\sum_{i+j = n} \left(J^*  \omega_i   Q  \omega_j + \omega_i  Q   \omega_j   J\right)
&= \sum_{j+k+l = n} \omega_k  Q  \omega_l  Q  \omega_j \;-\; \sum_{i+k+l = n} \omega_i  Q   \omega_k  Q   \omega_l =0.
\end{align*}
\end{proof}

In this way we may build a formal power series for a real $d_{L^*}$-closed form $\omega(t)$ with values in $L^*$, which satisfies \eqref{B condition}.  It remains to be shown that this series has a positive radius of convergence.  This is rather standard, by using Hodge theory as in \cite{KM}. Thus, let  $g$ be a Hermitian metric on $S$ (we can take for instance the lcK metric corresponding to $F$)  and $h$ a Hermitian metric on the holomorphic line bundle ${\cal L^*}$ (parallel with respect to the flat connection on $L^*$). Denote by $\bar \square_{\cal L^*} = \bar\del_{\cal L^*}\bar\del_{\cal L^*}^* + \bar\del_{\cal L^*}^*\bar\del_{\cal L^*}$  the resulting  Laplacian acting on smooth sections of $\wedge^{0,2}S\otimes {\cal L^*}$.  As $H^{0,2}(S, \cal L^*)=0$, $\bar \square_{\cal L^*}$ is invertible on $C^\infty(\wedge^{0,2}S\otimes {\cal L}^*)$ with inverse denoted by ${\mathbb G}$. Then,  letting
$$\beta_n = \bar \del^*_{\cal L^*} {\mathbb G}(\omega_n^{0,2}),$$
(so that $\beta_n$ manifestly solves \eqref{beta construction}) we get, for $n>1$,
$$\omega_n^{1,1} = \del_{\cal L^*} \bar \del^*_{\cal L^*} {\mathbb G}(\omega_n^{0,2})  + \textnormal{complex  conjugate},$$
and therefore 
$$\omega_n = \omega_n^{0,2} + \del_{\cal L^*} \bar \del^*_{\cal L^*}  {\mathbb G} (\omega_n^{0,2}) + \textnormal{complex  conjugate},$$
where $\omega_n^{0,2}$ is inductively defined by \eqref{term condition}. Schauder estimates for the Laplacian imply that in $C^{k,\alpha}(M)$ (for given $k\ge 2, 0<\alpha <1$)
\begin{equation*}
\|\omega_n\|_{k,\alpha} \leq C_{k,\alpha} \sum_{i+j=n} \|\omega_i\|_{k,\alpha} \|\omega_j\|_{k,\alpha},
\end{equation*}
for some positive constant $C_{k,\alpha}$. We can conclude that the power series $\omega(t) = \sum_{n=1}^{\infty} \omega_n t^n$ converges for small $t$,  by showing as in \cite[Chapter 4, Thm.~2.1]{KM} that $||\omega_n||_{k,\alpha} \le \frac{1}{16C_{k,\alpha}} S_n$, where $S_n= \frac{b_{k,\alpha}^n}{n^2}$ with $b_{k,\alpha}=16C_{k,\alpha}||F||_{k,\alpha}$ (therefore the series converges in $C^{k,\alpha}(M)$ for $t \in [0, 1/b_{k,\alpha})$).

In order to establish smoothness of $\omega(t)$,  we use elliptic regularity as in \cite{KM}. 
The real $2$-form $\omega(t)$ satisfies $\omega(0)\equiv 0, \dot{\omega}(0)\equiv F$  as well as  the equation  \eqref{e1} and
\begin{align}\label{e2}
\omega(t)^{1,1} = tF + 2 {\rm Re}\{\del_{\cal L^*} \bar \del^*_{\cal L^*} {\mathbb G} \omega(t)^{0,2}\}.
\end{align}
Substituting \eqref{e2} into \eqref{e1} and taking only the $(0,2)$ part for simplicity, we get
\begin{align*}
2i  \omega(t)^{0,2} = \omega(t)^{0,2} Q \omega(t)^{0,2} + \Big(\big(tF + 2{\rm Re}\{\del_{\cal L^*} \bar \del^*_{\cal L^*} {\mathbb G} \omega(t)^{0,2}\}\big) Q \big(tF +2{\rm Re}\{\del_{\cal L^*} \bar \del^*_{\cal L^*} {\mathbb G} \omega(t)^{0,2}\}\big)\Big)^{0,2}.
\end{align*}
Letting $\Theta_t = {\mathbb G} \omega(t)^{0,2}$, thus
\begin{align}\label{e3}
2i \bar\square_{\cal L^*} \Theta_t = E(\Theta_t),
\end{align}
where 
$$E (\Theta_t)= \Big(\big((\bar\square_{\cal L^*}  +2 {\rm Re}\circ \del_{\cal L^*} \bar \del^*_{\cal L^*})  \Theta_t + tF\big) Q \big((\bar\square_{\cal L^*}  +2 {\rm Re}\circ\del_{\cal L^*} \bar \del^*_{\cal L^*})  \Theta_t + tF\big)\Big)^{0,2}$$
is a non-linear second-order differential operator  with smooth coefficients acting on sections of $\wedge^{0,2}S\otimes {\cal L^*}$.  As $\Theta_0\equiv 0$, it follows that for small $t>0$, the non-linear equation \eqref{e3} is elliptic at $\Theta_t$,  so that $\Theta_t$ must be $C^{\infty}$ (see e.g. \cite[p.~467, Thm.~41]{Besse}).

\begin{rem} {\rm  Our method of proof  produces non-trivial generalized K\"ahler metrics  (or equivalently, non-trivial bi-Hermitian structures) in the case when $S=(M,J)$ is a compact complex surface of Kodaira dimension $-\infty$ with $b_1(M)$ even, endowed with a non-trivial section $\sigma$ of the anti-canonical bundle ${\cal K}^*_S$. This is the case when starting with a K\"ahler metric $F$,  the proof of Theorem~\ref{main} gives rise to a family of symplectic forms $\omega(t)$ satisfying the conditions of Proposition~\ref{marco} with  $\cal L = {\mathcal O}$.  Note that in this case one needs the analogue of Proposition~\ref{vanish} (i.e.  $H^{0,2}(S)=0$). This  is insured by using the Hodge isomorphism $H^{2,0}(S) \cong H^{0,2}(S)$ under the K\"ahler assumption on $S$,   and  the Serre duality $H^{2,0}(S)\cong H^0(S, {\cal K}_S)=0$ (as the Kodaira dimension is negative).  By \cite[Theorems~1 and 2]{AGG}, this 
allows to recast \cite[Theorem~6.2]{Goto-BH} entirely  within the framework of bi-Hermitian geometry. A more general approach to the deformation theory of generalized K\"ahler structures of any dimension  has been independently developed  by  M. Gualtieri and N. J. Hitchin~\cite{GH}.}
\end{rem}

\begin{rem} {\rm As observed in \cite{Goto-JDG}, as a by-product of Theorem~\ref{main} one obtains non-obstructness of the class $[Q F] \in H^1(S, T^{1,0}S)$, for any lcK metric $F$ with Lee form corresponding to $\cal L^*$,  should it exist. }
\end{rem}

\section{Towards a converse}
In order to further motivate the conjecture in the introduction, recall that by Proposition~\ref{BH},  any bi-Hermitian structure $(c,J_+, J_-)$  on $S=(M,J)$ with $J=J_+$ and   $J_+(x) \neq J_-(x)$ for each $x\in M$,   gives rise to a $d_{L^*}$-closed form $\omega$ whose $J$-invariant part is positive-definite: in other words,  $J$ is tamed by a locally conformally symplectic $2$-form $\omega$ with Lee form corresponding to $L^*$.  Note that the flat line bundle $L^*$ is the dual of the flat bundle $L$ for which $\sigma \in H^0(S, {\cal K}^*_S \otimes \cal L)$. As computed in \cite{A} (see also \cite[Eq.~(6)]{AD}), another necessary condition for $L^*$ is that 
${\rm deg}_g (\cal L^*)>0$ with respect to the Gauduchon metric of $(c, J)$, in particular $H^0(S, \cal L^\ell)=0$ for $\ell \ge 1$.  On a given minimal complex surface in the Class VII with a global spherical shell and second Betti number $b_2(M)>0$,  one can show that there  is a  finite number of such line bundles $L$.  We therefore ask the following more general
\begin{question} Let $S=(M,J)$ be a minimal compact complex surface in the class VII, with a global spherical shell, and $\cal L$ a flat holomorphic line bundle of real type such that $H^0(S, \cal L^{\ell})=0$ for any $\ell \ge 1$. Suppose there exists a $d_{L^*}$-closed $2$-form with values in $L^*$,  $\omega$,  whose $J$-invariant part is positive definite and denote $\Omega=[\omega]\in H^2(M, L^*)$ the corresponding Novikov cohomology class. Does $\Omega$ contain a  positive-definite $J$-invariant $2$-form $F$?
\end{question}
Note that the vanishing $H^0(S, \cal L^{\ell})=0$ for $\ell \ge 1$ is a necessary condition as the degree of $\cal L$ with respect to a Gauduchon metric $g$ in the conformal class of the lcK structure is ${\rm deg}_g (\cal L)=- \frac{1}{2\pi}\int_M |\theta|^2 dv_g <0$ (see e.g. \cite[Eq.~(5)]{AD}). 

For a compact complex surface with $b_1(S)$ even and $H^0(S, {\cal K}_S) = 0$ (and $\cal L= \cal O$), the analogous statement is known to be true as $H^2_{dR}(S) \cong H^{1,1}(S, {\mathbb R})$ and therefore the de Rham class $\Omega=[\omega]$ of a symplectic form taming $J$ defines a K\"ahler class by a result of Buchdahl~\cite{buchdahl} and Lamari~\cite{lamari}. 

A similar question has been raised in \cite[Remark 1]{B1} and \cite{OV2}. The general existence results in \cite{B1,B2} show that on a minimal Kato surface lcK metrics exist for  $L^*$ corresponding to an interval of {\it big}  Lee forms in  $H^1_{dR}(S)$ (i.e. for Lee forms with de Rham classes $t a$, $0\neq a \in H^1_{dR}(M), \ t>\varepsilon(S)\gg 0$) while the stability results in \cite{B1} or  \cite{G} combined with Remark~\ref{r:2}  imply that  the existence of lcK metrics in $H^2(M, L^*)$  is stable under complex deformations of $S$.

\section{Examples}~\label{Hopf}
In this section we give new examples of bi-Hermitian metrics on primary Hopf surfaces $S$, such that  $J_+(x) \neq J_-(x)$  on $S$ and $J_+= - J_-$ on an elliptic curve $E \subset S$. To the best of our knowledge,  these are new:  Indeed, they are not strongly bi-Hermitian (as $J_+ = - J_-$ on $E$) nor are they ASD  (see \cite{Pontecorvo}) as $J_+ \neq J_-$ everywhere.  According to \cite{CG}, these examples generate  bi-Hermitian structures of the same kind on the blow-ups of $S$ at $E$.

\vspace{0.2cm}

Recall that a {\it diagonal} primary Hopf surface $S$ is defined as the quotient of ${\mathbb C}^2 \setminus \{(0,0)\}$ by a contraction
\begin{equation}\label{contraction}
\g(z_1,z_2)=(a_1z_1,a_2z_2),\quad 0<|a_1| \le |a_2| <1.
\end{equation}
Letting $a=(a_1,a_2)$ we denote by $S_{a}$  the resulting diagonal Hopf surface. As $S_a \cong S^1 \times S^3$, any holomorphic line bundle is topologically trivial. The diagonal Hopf surface $S_a$ admits  two  elliptic curves, $E_1$ and $E_2$, which are respectively the projections of the axes $\{z_1=0\}$ and $\{z_2=0\}$ in ${\mathbb C^2}$ under the contraction \eqref{contraction}. A holomorphic section of the anti-canonical bundle ${\cal K}^*_{S_a}$ is induced by the $\gamma$-invariant bi-vector
$$z_1z_2 \frac{\del}{\del z_1}\wedge \frac{\del}{\del z_2},$$
showing that
$${\cal K}^*_{S_a} = [E_1 + E_2].$$
Without loss of generality, we  can choose the identification \eqref{H1} so that ${\cal K}^*_{S_a}$ corresponds to $a_1a_2 \in {\mathbb C}^*$ and thus $[E_i]$ corresponds to $a_i$. It follows that the flat bundles
$$\cal L_{p_1,p_2} \cong  p_1[E_1] + p_2[E_2], \ p_i \ge 0, $$
which correspond to $a_1^{p_1}a_2^{p_2}$ under \eqref{H1}, all admit holomorphic sections while ${\cal K}^*_{S_a} \otimes {\mathcal L}_{p_1, p_2}$ has a non-trivial sections for $p_i \ge -1$. Furthermore, it is not difficult to show (see e.g. \cite[Lemma 4]{AD}) that $\cal L_{p_1,p_2}$ is of real type if and only if $a_1^{p_1}a_2^{p_2}$ is a real number. Finally, the condition $H^0(S_a, {\cal L}_{p_1,p_2}^{\ell})=0$ implies $p_1\le -1$ or $p_2\le -1$. We thus get two families of flat bundles $\cal L_p = \cal L_{-1, p}, \ p\ge -1$ (resp. ${\tilde {\cal L}}_q= \cal L_{q, -1}, q \ge -1$) possibly satisfying the necessary conditions in Theorem~\ref{main}, subject to the constraint $a_2^p/a_1 \in {\mathbb R}_{>0}$ (resp. $a_1^q/a_2 \in {\mathbb R}_{>0}$). It remains to investigate whether or not there is a lcK metric on $S_a$ in $H^2(S_a, L^*_p)$ or $H^2(S_a, {\tilde L}^*_q)$.  

\vspace{0.2cm}
In the case $p=-1$ (or equivalently $q=-1$),  the existence of lcK structure in $H^2(S_a, {\cal L}^*_p)=H^2(S_a, {\cal K}^*_{S_a})$ (with $a_1a_2 \in {\mathbb R}$) is established in \cite{GO} while the existence of (strongly) bi-Hermitian deformations was observed in \cite{AD}. So we shall assume $p, q \ge 0$.

\vspace{0.2cm}
It is  shown in \cite{GO} that any diagonal Hopf surface $S_a$ admits a Vaisman lcK metric, i.e. a lcK Hermitian metric $g_0$ whose Lee form $\theta_0$ is parallel. As observed in \cite{AD}, the de Rham class $[\theta_0]$ of the Gauduchon--Ornea lcK metrics corresponds (via \eqref{H1}) to the real number $c_0=|a_1||a_2|$. The Vaisman lcK metrics always come in families, called $0$-type deformation in \cite{B}, 
$$g_t = g_0 + \frac{(t-1)}{|\theta_0|^2}(\theta_0\otimes \theta_0 + J\theta_0 \otimes J\theta_0), \ t>0, $$ 
with Lee forms $\theta_t = t \theta_0$, so that the de Rham class of $t\theta_0$ corresponds to $c_0^t$ (see e.g. \cite[Eq.~[7)]{B} or \cite[Eq.~(21) \& (22)]{AM}).  It follows that for each flat bundle of real type, ${\cal L}_{\mu}$,  corresponding (via \eqref{H1}) to a positive real number $\mu > 1$, there exists a lcK metric  with Lee form corresponding to $L_{\mu}^*$. Thus, if for some $p\ge 0$  we have $a_2^p/a_1 >1$ (resp. for some $q\ge 0$ we have $a_1^q/a_2 >1$), we can apply  Theorem~\ref{main} with ${\cal L}=\cal L_p$ (resp. $\cal L ={\tilde{ \cal L}}_q$) in order to construct bi-Hermitian metrics metrics on $S^1 \times S^3$.  As a special case, we can take $a_1=a_2= \lambda \in ]0,1[$ and $\cal L =\cal L_0 \cong {\tilde {\cal L}}_0$ (i.e. $p=0=q$).

\vspace{0.2cm}
Another class of (primary) Hopf surfaces are the non-diagonal ones, when $S=S_{b, \lambda, m}$ is obtained as a quotient of ${\mathbb C}^2 \setminus \{(0,0)\}$ by the contraction
$$\gamma(z_1, z_2)= (b^m z_1 + \lambda z_2^m, bz_2),  \ 0<|b|<1, \lambda \neq 0, m \ge 1.$$ 
The deformation argument in \cite{GO} shows that for any $\mu>1$  such $S$ still admits  lcK metrics with fundamental $2$-form  in $H^2(S, L_{\mu}^*)$,   where $L_{\mu}$ is a flat bundle corresponding (via \eqref{H1}) to $\mu$.
Furthermore,   the axis $z_2=0$ of ${\mathbb C}^2$ defines an elliptic curve $E \subset S$ with ${\cal K}^*_S \cong (m+1) E$ and corresponding complex number $b^{m+1}$. If $b\in ]0,1[$ is real, then Theorem~\ref{main} applies for the flat bundles $\cal L =\cal L_p \cong - p[E], \ 1 \le p \le m+1$,  corresponding to $b^{-p}$. 

\begin{rem} {\rm It is well-known (see e.g. \cite{OV}) that the Vaisman lcK metrics we have used to produce our examples of bi-Hermitian metrics admit potentials, i.e. there exists a smooth (real) section $f$ of $L^*$ such that $F= 2i\del_{\cal L^*} \bar \del_{\cal L^*} f= d_{L^*} d^c_{L^*} f$. This allows to construct  bi-Hermitian metrics via a hamiltonian flow, originally due to Hitchin~\cite{hit},  and re-casted in the case of Hopf surfaces in our previous work~\cite{AD} (in order  to obtain  strongly bi-Hermitian metrics).  Indeed, for any such potential $f$, $X_f= Q(df)$ is a smooth vector field on $M$  whose flow $\varphi_s$ defines a family of $2$-forms
\begin{equation*}
\omega(t) := \int_{0}^t \varphi_s^*(d_{L^*}d_{L^*}^c f) ds,
\end{equation*}
satisfying the second relation in Proposition~\ref{twisted lemma of Gualtieri} and  $\Big(\frac{d \omega(t)}{dt}\Big)_{t=0}  = d_{L^*}d_{L^*}^c f =F$ (see \cite{Gual}). Thus, we obtain a familly of bi-Hermitian structures with  $J_+=J$ and $J_-^t = -\varphi_t (J)$}. \end{rem}

\appendix

\section{A rough classification of bi-Hermitian complex surfaces in the class ${\rm VII}_0$}
In this section,  we recast the list obtained in \cite{Dl} of the minimal compact complex surfaces $S=(M,J)$ in the Kodaira class VII,  possibly admitting compatible  bi-Hermitian structures,    in terms of the classification of bi-Hermitian structures in the three classes (i)--(iii) from the introduction (defined as a function of the number of connected components of the divisor determined by $\sigma=[J_+, J_-]^{\sharp} \in H^0(S, {\cal K}^*_S\otimes \cal L)$). 

As in \cite{Dl},  we will assume that the algebraic dimension of $S$ is zero, i.e. that there are no non-constant meromorphic functions on $S$, and that the fundamental group $\pi_1(S) \cong {\mathbb Z}$. These assumptions exclude only the cases of elliptic primary Hopf surfaces (i.e. diagonal Hopf surfaces  with $a_1^{p_1} = a_2^{p_2}$ for some $p_1,p_2 \in {\mathbb Z}$, see \cite{kodaira}), and the secondary Hopf surfaces (which are finitely covered by a primary Hopf surface described in Section~\ref{Hopf}). We then have the following

\begin{prop} Let $S=(M,J)$ be a minimal compact complex surface in the Kodaira class VII  of algebraic dimension $0$, endowed with a compatible bi-Hermitian structure $(c, J_+=J, J_-)$.  By replacing $S$ with a finite covering if necessary, assume also the fundamental group of $S$ is ${\mathbb Z}$. Then one of the following must hold.
\begin{enumerate}
\item[\rm(i)] Everywhere on $M$, $J_+ \neq J_-$ and $J_+\neq -J_-$. Then $\sigma =[J_+,J_-]^{\sharp} \in H^0(S, {\cal K}^*_S \otimes {\cal K}_S)$ never vanishes, and $S$ is a primary Hopf surface described in \cite[Thm.~1]{AD}.
\item[\rm (ii)] Everywhere on $M$, $J_+ \neq J_-$ but for at least one $x \in M$, $J_+(x) = -J_-(x)$. Then $\sigma =[J_+,J_-]^{\sharp} \in H^0(S, {\cal K}^*_S \otimes \cal L)$ with $\cal L \neq \cal O$, and $S$ must be either a primary Hopf surface, a parabolic Inoue surface, or a surface with GSS of intermediate type.
\item[\rm (iii)] There are points on $M$ where $J_+ = J_-$ and also points where $J_+ = - J_-$. Then $\sigma =[J_+,J_-]^{\sharp} \in H^0(S, {\cal K}^*_S )$ and $S$ must be a primary Hopf surface, a parabolic Inoue surface or an even Inoue-Hirzebruch surface.
\end{enumerate}
\end{prop}
\begin{proof}
The case (i) is treated in \cite{AD}. 

We will next establish (iii). As the zero set of the holomorphic section $\sigma = [J_+, J_-]^{\sharp} \in H^0(S, {\cal K}^*_S \otimes {\mathcal L})$ consists of the points where either $J_+(x)=J_-(x)$ or $J_+(x)=-J_-(x)$ (see e.g. \cite{AGG}), we conclude that  ${\cal K}^*_S \otimes {\mathcal L}$ is represented by an effective divisor $D$ with at least two connected components. By \cite[Thm.~0.4]{Dl}, $S$ must be then either a primary Hopf surface, a (parabolic) Inoue surface, or an Inoue--Hirzebruch surface.  The Inoue--Hirzebruch surfaces come in two families, called {\it even} or {\it odd} in \cite{Dl}, and it has been already observed in the proof of \cite[Cor.~3.45]{Dl} that the odd Inoue--Hirzebruch surfaces cannot appear as the effective divisor  representing ${\cal K}^*_{S}\otimes {\mathcal L}$ for some flat bundle $\cal L$ is connected (it is given by one cycle of rational curves). It remains therefore to show that in the three cases for $S$,  we must have ${\cal L} = \cal O$. Suppose for contradiction that $\cal L \neq \cal O$: by the degree computation in \cite{A}, we must then have $H^0(S, {\cal L}^{\ell})=0$ for $\ell \ge 1$.

If $S$ is a diagonal Hopf surface as in Section~\ref{Hopf} with $[D] \cong {\cal K}^*_S \otimes {\cal L}$ and $H^0(S, {\cal L}^{\ell})=0$ for $\ell \ge 1$, we have already noticed that $[D] \cong (p+1) [E_2],  p \ge 0$ or $[D] \cong (q+1)[E_1], q \ge 0$. As we assume that $S$  doesn't have non-trivial meromorphic functions, it follows $D= (p+1)E_2$ or $D= (q+1)E_1$, a contradiction as $D$ has at least two connected components. Similarly,  if $S$ is a non-diagonal primary Hopf surface, then ${\cal K}^*_S \cong (m+1)[E]$ and therefore $D = (p+1)E, p\ge 0$, a contradiction.

If $S$ is a parabolic Inoue surface or an even Inoue-Hirzebruch surface, then ${\cal K}^*_{S} \cong [A + B]$ for a cycle $A$ of rational curves and a smooth elliptic curve $B$,  or for two cycles of rational curves $A,B$, respectively (see \cite{Nakamura} and \cite[Prop.~2.27]{Dl}). In the first case,  the cycle $A$ represents a flat bundle, and  therefore, by \cite[Lemma 2.26]{Dl},  we must have $D= B$, a contradiction as $B$ is connected; in the second case, neither $A$ nor $B$ represents a flat bundle because $A^2<0$ and $B^2<0$ (see  \cite[Thm.~6.1]{Nakamura} or \cite[Cor.~2.28]{Dl84}) so we obtain a contradiction as we assumed ${\mathcal L} \neq \cal O$.

Finally, we consider the case (ii): the only additional point with respect to \cite[Thm.~0.4]{Dl} is that  Inoue--Hirzebruch surfaces cannot support bi-Hermitian structures of the class (ii). Indeed,  as we have already explained in the introduction, a necessary condition for the existence of such bi-Hermitian structures  is that $H^0(S, {\cal K}^*_S \otimes {\cal L}) \neq 0$ for a non-trivial flat bundle with $H^0(S, {\cal L}^{\ell})=0$ for $\ell \ge 1$. For the even Inoue--Hirzebruch surfaces the only flat bundle ${\mathcal L}$  with $H^0(S, {\cal K}^*_S \otimes \cal L) \neq 0$ is the trivial one, see \cite[Prop.~2.14]{Dl84}, while for the odd  Inoue--Hirzebruch surfaces,  $H^0(S, {\cal K}^*_S \otimes \cal L) \neq 0$  for a (unique) flat bundle ${\mathcal L}$ which satisfies ${\mathcal L}^2 \cong \cal O$  by \cite[Prop.~2.14]{Dl84} and \cite[Lemma~2.5]{Nakamura}. \end{proof}

\begin{rem} {\rm The general existence problem for bi-Hermitian structures can be reduced to the minimal case by \cite[Lemma~3.43]{Dl} and the construction in \cite{CG}. From the list above, the existence is now fully established in the case (i) by \cite{AD} and in the case (iii) by \cite{FP}.  The construction in this paper provides the first existence results in the case (ii), but a complete resolution is still to come. We also note (see \cite[Prop.~3]{AG}) that the bi-Hermitian minimal complex surfaces in the class (iii) are precisely the ones arising from  twisted generalized K\"ahler structures. }
\end{rem}


\begin{thebibliography}{AUT}
\bibitem{A} {V. Apostolov},  {\it Bihermitian surfaces with odd first Betti number}, {Math. Z. {\bf 238} (2001), 555--568}.

\bibitem{AD} {V. Apostolov and G. Dloussky},   {\it Bihermitian metrics on Hopf surfaces},  {Math. Res. Lett. {\bf 15} (2008),  827--839}.

\bibitem{AGG} {V. Apostolov, P. Gauduchon, and G. Grantcharov}, {\it Bihermitian structures on complex surfaces}, Proc. London Math. Soc. (3) {\bf 79} (1999), 414--428. {\it Corrigendum},  {\bf 92} (2006), 200--202.

\bibitem{AG} V. Apostolov, M. Gualtieri, {\it Generalized K\"ahler manifolds, commuting complex structures and split tangent bundles}, Commun. Math. Phys. {\bf 271} (2007), 561--575.

\bibitem{AM} V.~Apostolov and O.~Mu\u{s}karov, {\it Weakly Einstein Hermitian surfaces}, Ann. Inst.  Fourier (Grenoble) {\bf 49} (1999), 1673--1692.


\bibitem{BM} C.~Bartocci and E. Macr{\`i},  {\it Classification of Poisson surfaces}, Comm. Contemp. Math. {\bf 7} (2005), 89--95.

\bibitem{bpv} W. Barth, K.~Hulek, C. Peters, and A. Van de Ven, {\rm Compact Complex Surfaces},  {\it Springer, Heidelberg, Second Edition, 2004}.

\bibitem{B}{F. Belgun,} {\it On the metric structure of non-K\"ahler complex surfaces}, {Math. Ann. {\bf 317} (2000), 1--40}.

\bibitem{Besse} A. L. Besse, {\it Einstein manifolds},
Ergeb. Math. Grenzgeb. {\bf 3}, Springer-Verlag, Berlin, Heidelberg, New York, 1987.

\bibitem{B1}{M. Brunella,} {\it Locally conformally K\"ahler metrics on certain non-K\"ahlerian surfaces},   {Math. Ann. (2010) {\bf 346}  629--639}.

\bibitem{B2}{M. Brunella,} {\it Locally conformally K\"ahler metrics on Kato surfaces},   {Nagoya Math. J. {\bf 202}, 77--81}.

\bibitem{buchdahl} N.~ Buchdahl, {\it On compact K\"ahler surfaces}, Ann. Inst. Fourier {\bf 49} (1999), 287--302.



\bibitem{CG} {G. Cavalcanti and M. Gualtieri}, {\it Blowing up generalized K\"ahler $4$-manifolds},  Bull. Braz. Math. soc. New Sreies {\bf 42}, 537--557.

\bibitem{Dl84} G.~Dloussky, {\it Structure des surfaces de Kato}, M\'emoire de la SMF {\bf 14} (1984).

\bibitem{Dl} G.~Dloussky, {\it On surfaces of class $VII_0^+$ with numerically anticanonical divisor}, Amer. J. Math. {\bf 128} (2006), 639--670.


\bibitem{DOT3} {G. Dloussky, K. Oeljeklaus, and M.Toma}, {\it Class VII$_{0}$  surfaces with $b_{2}$ curves},  {Tohoku Math. J. {\bf 55} (2003), 283--309.}

\bibitem{DO} S. Dragomir and L. Ornea, Locally coformal K\"ahler geometry, {\it Birkhauser 1998}.
 
 
 \bibitem{FP} A.~Fujiki, M.~Pontecorvo {\it Anti-self-dual bihermitian structures on Inoue surfaces}, J. Differential Geom. {\bf 85} (2010), 15--71.


\bibitem{G}{R. Goto} {\it On the stability of locally conformal K\"ahler structures},  {available at arxiv 1012.2285v1 (2010)}.


\bibitem{Goto-BH} {R. Goto}, {\it Unobstructed K-deformations of generalized complex structures and bi-Hermitian structures}, {Adv. Math. {\bf 231} (2012), 1041--1067}.

\bibitem{Goto-JDG} R. Goto, {\it Deformations of generalized complex and generalized K\"ahler structures}, J. Differential Geom. {\bf 84} (2010),  525--560.



\bibitem{GO}{P. Gauduchon and L. Ornea} {\it Locally conformally K\"ahler metrics on Hopf surfaces}, {Ann. Inst. Fourier (Grenoble), {\bf 48} (1998), 1107–1127}.

\bibitem{gualtieri-GK} {M. Gualtieri}, {\it Generalized K\"ahler geometry}, available at  arXiv:1007.3485.

\bibitem{Gual} {M. Gualtieri}, {\it Branes on Poisson varieties},  in ``The Many facets of Geometry: A Tribute to Nigel Hitchin'' (eds. J.-P. Bourguignon, O. Garcia-Prada and S. Salamon), Oxford University Press,  2010.

\bibitem{GH} M. Gualtieri, N.~J.~Hitchin, private communication.

\bibitem{hit} {N. J. Hitchin}, {\it  Bihermitian metrics on del Pezzo surfaces}, J. Symplectic Geom. {\bf 5} (2007), 1--8.

\bibitem{Hitchin-deform} N. J. Hitchin, {\it  Deformations of holomorphic Poisson manifolds},  Mosc. Math. J. {\bf 12} (2012),  567--591, available at  arXiv:1105.4775.

\bibitem{hitchin0} N.~J.~Hitchin, {\it Generalized Calabi-Yau manifolds}, Q. J. Math. {\bf 54} (2003), 281--308.

\bibitem{kodaira}  K.~Kodaira, {\it On the structure of complex analytic surfaces II, III}, Am. J. Math {\bf 86} (1966), 751--798; Am. J. Math. {\bf 88} (1966), 682--721.


\bibitem{lamari} A.~Lamari, {\it Courants k\"ahl\'eriens et surfaces compactes}, Ann. Inst. Fourier {\bf 49} (1999), 263--285.

 \bibitem{KM} J. Morrow, K.~Kodaira, Complex manifolds. {\it Holt, Rinehart and Winston, Inc., New York-Montreal, Que.-London}, 1971.
 
 \bibitem{Nakamura} I. Nakamura, {\it On surfaces of class $VII_0$ with curves}, Invent. math. {\bf 78} (1984), 393--343.
 
\bibitem{OV}  L. Ornea and M. Verbitsky  {\it A report on locally conformally K\"ahler manifolds. Harmonic maps and differential geometry}, 135--149, Contemp. Math. {\bf 542}, Amer. Math. Soc., Providence, RI, 2011. 

\bibitem{OV2}  L.~Ornea and M.~Verbitsky, {\it Morse-Novikov cohomology of locally conformally K\"hler manifolds},  J. Geom. Phys. {\bf 59} (2009),  295--305.

\bibitem{Pontecorvo} M. Pontecorvo, {\it Complex structures on Riemannian $4$-manifolds}, Math. Ann. {\bf 309} (1997), 159--177.


\bibitem{siu} Y.-T. Siu, {\it Every $K3$ surface is K\"ahler}, Invent. Math. {\bf  73} (1983), 139--150.

\bibitem{todorov} A.~Todorov, {\it Applications of the K\"ahler-Einstein-Calabi-Yau metric to moduli of $K3$ surfaces}, Invent. Math. {\bf 61} (1980), 251--265.



\bibitem{Vaisman} I. Vaisman, {\it Conformal changes of generalized complex structures}, An. Sti. Univ. ``Al. I. Cuza'' Iasi, Sect. I, Mat. (N.S.), {\bf 54} (2008), 1-14.

\end{thebibliography}
\end{document}